\documentclass[12pt,a4paper]{article}

\usepackage{amsfonts}
\usepackage{amsmath}
\usepackage{amsbsy}
\usepackage{amsxtra}
\usepackage{latexsym}
\usepackage{amssymb}
\usepackage[active]{srcltx}
\usepackage{bbm}

\usepackage{algorithm, algorithmic}
\usepackage{color,fancybox,graphicx}
\DeclareMathOperator{\RSS}{RSS}
\DeclareMathOperator{\trace}{tr}

\DeclareMathOperator*{\argmin}{argmin}
\def\P{\mathbb P}

\newcommand{\un}{{\mathbbm{1}}}
\newcommand{\E}{\mathbb E}
\newcommand{\iso}{\mathrm{iso}}
\newcommand{\anti}{\mathrm{anti}}

\newcommand{\Ve}{{\mathrm{Vect}}}

\newenvironment{proof}{{\bf Proof.}}{\hfill$\Box$\\}
\newtheorem{theorem}{Theorem}

\newtheorem{proposition}{Proposition}

\newcommand{\R}{\mathbb R}
\newcommand{\N}{\mathbb N}

\newtheorem{lem}{Lemma}

\begin{document}
\begin{center}

{\sc \Large Iterative Isotonic Regression \\
\vspace{0.2cm}
}

\vspace{0.5cm}

Arnaud GUYADER\footnote{Corresponding author.}\\
Universit\'e Rennes 2, INRIA and IRMAR \\
Campus de Villejean, Rennes, France\\
\textsf{arnaud.guyader@uhb.fr}
\vspace{0.5cm}

Nick HENGARTNER\\
Los Alamos National Laboratory\\
NM 87545, USA\\
\textsf{nickh@lanl.gov}
\vspace{0.5cm}

Nicolas J\'EGOU\\
Universit\'e Rennes 2\\
Campus de Villejean, Rennes, France\\
\textsf{nicolas.jegou@uhb.fr}
\vspace{0.5cm}

Eric MATZNER-L\O BER\\
Universit\'e Rennes 2\\
Campus de Villejean, Rennes, France\\
\textsf{eml@uhb.fr}

\end{center}

\begin{abstract}
\noindent {\rm This article introduces a new nonparametric method for estimating a univariate regression function of bounded variation. The method exploits the Jordan decomposition which states that a function of bounded variation can be decomposed as the sum of a non-decreasing function and a non-increasing function.  This suggests combining the backfitting algorithm for estimating additive functions with isotonic regression for estimating monotone functions.  The resulting iterative algorithm is called Iterative Isotonic Regression (\textsf{I.I.R.}). 
The main technical result in this paper is the consistency of the proposed estimator
when the number of iterations $k_n$ grows appropriately with the sample 
size $n$. The proof requires two auxiliary results that are of interest in and by themselves: firstly, we generalize the well-known consistency property of isotonic regression to the framework of a non-monotone regression function, and secondly, we relate the backfitting algorithm to Von Neumann's algorithm in convex analysis. 

\medskip


\medskip

}

\end{abstract}

\section{Introduction}
Consider the  regression model 
\begin{equation}\label{eq:general-model}
Y=r(X)+\varepsilon 
\end{equation}  
where $X$ and $Y$ are real-valued random variables, with $X$ distributed according to a non-atomic law $\mu$ on $[0,1]$, $\E\left[Y^2\right]<\infty$ and $\E\left[\varepsilon|X\right]=0$. We want to estimate the regression function $r$, assuming it is of bounded variation. Since $\mu$ is non-atomic, we will further assume, without loss of generality, that $r$ is right-continuous. The Jordan decomposition states that $r$ can be written as the sum of a non-decreasing function $u$ and a non-increasing function $b$
\begin{equation}\label{eq:jordan-decomposition}
r(x)=u(x)+b(x).
\end{equation}
The underlying idea of the estimator that we introduce in this paper consists in viewing this decomposition as an additive model involving the increasing and the decreasing parts of $r$. This leads us to propose an ``Iterative Isotonic Regression'' estimator (abbreviated to \textsf{I.I.R.}) that combines the isotonic regression and backfitting algorithms, two well-established algorithms for  estimating monotone functions and  additive models, respectively.\\

\noindent The Jordan decomposition (\ref{eq:jordan-decomposition}) is not unique in general. However, if one requires that both terms on the right-hand side have singular associated Stieltjes measures and that 
\begin{equation}\label{iasbc}
\int_{[0,1]}r(x)\mu(dx)=\int_{[0,1]}u(x)\mu(dx),
\end{equation}
then the decomposition is unique and the model is identifiable. Let us emphasize that, from a statistical point of view, our assumption on $r$ is mild. The classical counterexample of a function that is not of bounded variation is $r(x)=\sin(1/x)$ for $x\in(0,1]$, with $r(0)=0$. \\

\noindent Estimating a monotone regression function is the archetypical shape restriction estimation problem. Specifically, assume that the regression function $r$ in (\ref{eq:general-model}) is non-decreasing, and suppose we are given a sample ${\cal D}_n=\{(X_1,Y_1),\ldots,(X_n,Y_n)\}$  of i.i.d. $\R\times\R$ valued random variables distributed as a generic pair $(X,Y)$. Then denote $x_1=X_{(1)}<\ldots< x_n=X_{(n)},$ the ordered sample and $y_1,\ldots, y_n$ the corresponding observations. In this framework, the Pool-Adjacent-Violators Algorithm (PAVA) determines a collection of non-decreasing level sets solution to the least square minimization problem
\begin{equation}\label{eq:min1}
\min_{u_1\leq \ldots \leq u_n} \frac{1}{n}\sum_{i=1}^n \left(y_i-u_i\right)^2.
\end{equation} 
These estimators have raised great interest in the literature for decades since they are nonparametric, data driven and easy to implement. Early work on the maximum likelihood estimators of distribution parameters subject to order restriction date back to the 50's, starting with Ayer \textit{et al.} \cite{ayer1955empirical} and Brunk \cite{brunk1955maximum}.  Comprehensive treatises on isotonic regression include Barlow \textit{et al.} \cite{barlow1972statistical} and Robertson \textit{et al.} \cite{robertson1988order}. For improvements and extensions of the PAVA approach to more general order restrictions, see Best and Chakravarti \cite{BestChakravarti1990}, Dykstra \cite{dykstra1981isotonic}, and Lee \cite{lee1983min}, among others.\\

\noindent The solution of (\ref{eq:min1}) can be seen as the metric projection, with respect to the Euclidean norm, of the vector $y=(y_1,\ldots,y_n)$  on the isotone cone ${\cal C}^+_n$
\begin{equation}\label{eq:isotone-cone}
{\cal C}^+_n=\left\{u=\left(u_1,\ldots,u_n\right)\in \R^n:u_1\leq \ldots\leq u_n\right\}.
\end{equation}
That projection is not linear, which is the reason why analyzing these estimators is technically challenging.\\

Interestingly, one can interpret the isotonic regression estimator as the slope of a convex approximation of the primitive integral of $r$. This leads to an explicit relation between $y$ and the vector of the adjusted values, known as the ``min-max formulas'' (see Anevski and Soulier \cite{anevski2011monotone} for a rigorous justification). This point of view plays a key role in the study of the asymptotic behavior of isotonic regression. The consistency of the estimator was established by Brunk \cite{brunk1955maximum} and Hanson \textit{et al.} \cite{hanson1973consistency}. Brunk \cite{brunk1970estimation} proved its cube-root convergence at a fixed point and obtained the pointwise asymptotic distribution, and Durot \cite{durot2007error} provided a central limit theorem for the $L_{p}$-error.\\

\noindent Let us now discuss the additive aspect of the model. In a multivariate setting, the additive model was originally suggested by Friedman and Stuetzle \cite{friedman1981projection} and popularized by Hastie and Tibshirani \cite{hastie1990generalized} as a way to accommodate the so-called curse of dimensionality. The underlying idea of additive models is to approximate a high dimension regression function $r:\R^d \rightarrow \R$ by a sum of one-dimensional univariate functions, that is
\begin{equation}\label{eq:additive-models}
r(\textbf{X})=\sum_{j=1}^dr_j(X^{j}).
\end{equation}
Not only do additive models provide a logical extension of the standard linear regression model which facilitates the interpretation, but they also achieve  optimal rates of convergence that do not depend on the dimension $d$ (see Stone \cite{stone1985additive}).\\ 

\noindent Buja \textit{et al.} \cite{buja1989linear} proposed the backfitting algorithm as a practical method for estimating additive models. It consists in iteratively fitting the partial residuals from earlier steps until convergence is achieved. Specifically, if the current estimates are $\hat{r}_1,\ldots, \hat{r}_d$, then $\hat{r}_j$ is updated by smoothing $y-\sum_{k\neq j} \hat{r}_k$ against $X^{j}$. The backfitted estimators have mainly been studied in the case of linear smoothers. H{\"a}rdle and Hall \cite{hardle1993backfitting} showed that when all the smoothers are orthogonal projections, the whole algorithm can be replaced by a global projection operator. Opsomer and Ruppert \cite{opsomer1997fitting}, and Opsomer \cite{opsomer2000asymptotic}, gave asymptotic bias and variance expressions in the context of additive models fitted by local polynomial regression. Mammen, Linton and Nielsen \cite{mammen1999existence} improved these results by deriving a backfitting procedure that achieves the oracle efficiency (that is, each component can be estimated as well as if the other components were known). This procedure was extended  to several different one-dimensional smoothers including kernel, local polynomials and splines by Horowitz, Klemel{\"a} and Mammen \cite{horowitz2006optimal}. Alternative estimation procedures for additive models have been considered by Kim, Linton and Hengartner \cite{klh}, and by Hengartner and Sperlich \cite{hs}.\\ 

\noindent In the present context, we propose to apply the backfitting algorithm to decompose a univariate function by alternating isotonic and antitonic regressions on the partial residuals in order to estimate the additive components $u$ and $b$ of the Jordan decomposition (\ref{eq:jordan-decomposition}). The finite sample behavior of this estimator has been studied in a related paper by Guyader {\it et al.} (see \cite{guyader2012}). Among other results, it is stated that the sequence of estimators obtained in this way converges to an interpolant of the raw data (see section \ref{sec:iir} below for details).\\  %Although this overfitting phenomenon is obviously not a desirable property from a statistical point of view, it furnishes a convergence result of the backfitting algorithm in a non-linear case.\\

\noindent Backfitted estimators in a non-linear case have also been studied by Mammen and Yu \cite{mammen2007additive}. Specifically, assuming that the regression function $r$ in (\ref{eq:additive-models}) is an additive function of isotonic one-dimensional functions $r_j$, they estimate each additive component by iterating the PAVA in a backfitting fashion. %Thus our estimator is rather comparable to theirs when $d=2$, and indeed there exist similarities when studying the finite sample case (see Guyader \textit{et al.} \cite{guyader2012}, Remark 3 in section 3.1).
 Moreover, Mammen and Yu show that, as in the linear case, their estimator achieves the oracle efficiency and, in each direction, they recover the limit distribution exhibited by Brunk \cite{brunk1970estimation}.\\

\noindent The main result addressed in this paper states the consistency of our \textsf{I.I.R.} estimator. Denoting $\hat{r}_n^{(k)}$ the Iterative Isotonic Regression estimator resulting from $k$ iterations of the algorithm, we prove the existence of a sequence of iterations $(k_n)$, increasing with the sample size $n$, such that 
\begin{equation}\label{eq:existing-sequence1}
\E\left[\|\hat{r}_n^{(k_n)}-r\|^2\right]\underset{n\rightarrow \infty}{\longrightarrow}0
\end{equation}
where $\|.\|$ is the quadratic norm with respect to the law $\mu$ of $X$. Our analysis identifies two error terms: an estimation error that comes from the isotonic regression, and an approximation error that is governed by the number of iterations $k$.\\

\noindent Concerning the estimation error, we wish to emphasize that all asymptotic results about isotonic regression mentioned above assume monotonicity of the regression function $r$. In our context, at each stage of the iterative process, we apply an isotonic regression to an arbitrary function (of bounded variation). As a result, we prove in Section \ref{sec:isotonic-consistency} the $L_2(\mu)$ consistency of isotonic regression for the metric projection of $r$ onto the cone of increasing functions (see Theorem \ref{thm:iso-consistency}). \\

\noindent  The approximation term can be controlled by increasing the number of iterations.  This is made possible thanks to the interpretation of \textsf{I.I.R.} as a Von Neumann's algorithm, and by applying related results in convex analysis (see Proposition \ref{pro:approximation}). Putting estimation and approximation errors together finally leads to the consistency result (\ref{eq:existing-sequence1}).\\

\noindent Let us remark that, as far as we know, rates of convergence of Von Neumann's algorithm have not yet been studied in the context of  bounded variation functions. Hence, at this time, it seems difficult to establish rates of convergence for our estimator without further restrictions on the shape of the underlying regression function. Thus, the results we present here may be considered as a starting point in the study of novel methods which would consist in applying isotonic regression with no particular shape assumption on the regression function.\\

\noindent The remainder of the paper is organised as follows. We first give  further details and notations about the construction of \textsf{I.I.R.} in Section \ref{sec:iir}. The general consistency result for isotonic regression is given in Section \ref{sec:isotonic-consistency}. The main result of this article, the consistency of \textsf{I.I.R.}, is established in Section \ref{sec:iir-consistency}. Most of the proofs are postponed to Section \ref{sec:appendix}, while related technical results are gathered in Section \ref{tech}.          

\section{The \textsf{I.I.R.} procedure}\label{sec:iir}

Denote by $y=(y_1,\ldots,y_n)$ the vector of observations corresponding to the ordered sample $x_1=X_{(1)}<\ldots<X_{(n)}=x_n$. We implicitly assume in this writing that the law $\mu$ of $X$ has no atoms. We denote by $\iso(y)$ (resp. $\anti(y)$) the metric projection of $y$ with respect to the Euclidean norm onto the isotone cone ${\cal C}_n^+$ (resp. ${\cal C}_n^-=-{\cal C}_n^+$) defined in (\ref{eq:isotone-cone}):   
\begin{align*}
\iso(y)=\argmin_{u\in{\cal C}_n^+}\frac{1}{n}\sum_{i=1}^n\left(y_i-u_i\right)^2=\argmin_{u\in{\cal C}_n^+}\|y-u\|^2_n\\
\anti(y)=\argmin_{b\in{\cal C}_n^-}\frac{1}{n}\sum_{i=1}^n\left(y_i-b_i\right)^2=\argmin_{b\in{\cal C}_n^-}\|y-b\|^2_n.
\end{align*}
The backfitting algorithm consists in updating each component by smoothing the partial residuals, i.e., the residuals resulting from the current estimate in the other direction. Thus the Iterative Isotonic Regression algorithm goes like this: 
\begin{algorithm}[H]                      
\caption{Iterative Isotonic Regression (\textsf{I.I.R.})}          
\label{algo:iir}                           
\begin{algorithmic}                   
\STATE (1) Initialization: $\hat{b}^{(0)}_n=\left(\hat{b}^{(0)}_1[1],\ldots,\hat{b}^{(0)}_n[n]\right)=0$
\STATE (2) Cycle: for $k\geq 1$\\
\begin{equation*}\label{eq:cycle}
\begin{array}{ll}
\hat{u}^{(k)}_n&=\iso\left(y-\hat{b}_n^{(k-1)}\right)\\
\hat{b}^{(k)}_n&=\anti\left(y-\hat{u}_n^{(k)}\right)\\
\hat{r}^{(k)}_n&=\hat{u}_n^{(k)}+\hat{b}_n^{(k)}.
\end{array}
\end{equation*}
\STATE (3) Iterate (2) until a stopping condition to be specified is achieved.   
\end{algorithmic}
\end{algorithm}

\noindent Guyader \textit{et al.} \cite{guyader2012} prove that the terms of the decomposition $\hat{r}_n^{(k)}=\hat{u}_n^{(k)}+\hat{b}_n^{(k)}$ have singular Stieltjes measures. Furthermore, by starting with isotonic regression, the terms $\hat{u}_n^{(k)}$ have all the same empirical mean as the original data $y$, while all the $\hat{b}_n^{(k)}$ are centered. Hence, for each $k$, the decomposition  $\hat{r}_n^{(k)}=\hat{u}_n^{(k)}+\hat{b}_n^{(k)}$ satisfies the condition (\ref{iasbc}), and that decomposition is unique (identifiable).\\

\noindent Algorithm \ref{algo:iir} furnishes vectors of adjusted values. In the following, we will consider one-to-one mappings between such vectors and piecewise functions defined on the interval $[0,1]$. For example, the vector $\hat{u}_n^{(k)}=(\hat{u}_n^{(k)}[1],\ldots,\hat{u}_n^{(k)}[n])$ is associated to the real-valued function $\hat{u}_n^{(k)}$ defined on $[0,1]$ by 
\begin{equation}\label{eq:prolongement}
\hat{u}_n^{(k)}(x)=\hat{u}_n^{(k)}[1]\un_{[0,X_{(2)})}(x)+\sum_{i=2}^{n-1}\hat{u}_n^{(k)}[i]\un_{[X_{(i)},X_{(i+1)})}(x)+\hat{u}_n^{(k)}[n]\un_{[X_{(n)},1]}(x).
\end{equation} 
Observe that our definition of $\hat{u}_n^{(k)}(x)$ makes it right-continuous. Obviously, equivalent formulations hold for $\hat{b}_n^{(k)}$ and $\hat{r}_n^{(k)}$ as well.\\

\noindent Figure \ref{fig:iir-interpolation} illustrates the application of \textsf{I.I.R.} on an example. The top left-hand side displays the regression function $r$, and $n=100$ points $(x_i,y_i)$, with $y_i=r(x_i)+\varepsilon_i$, where the $\varepsilon_i$'s are Gaussian centered random variables. The three other figures show the estimations $\hat{r}_n^{(k)}$ obtained on this sample for $k=1,10$, and $1,000$ iterations. According to (\ref{eq:prolongement}), our method fits a piecewise constant function. Moreover, increasing the number of iterations tends to increase the number of jumps. 

\begin{figure}[H]
\begin{center}
\input{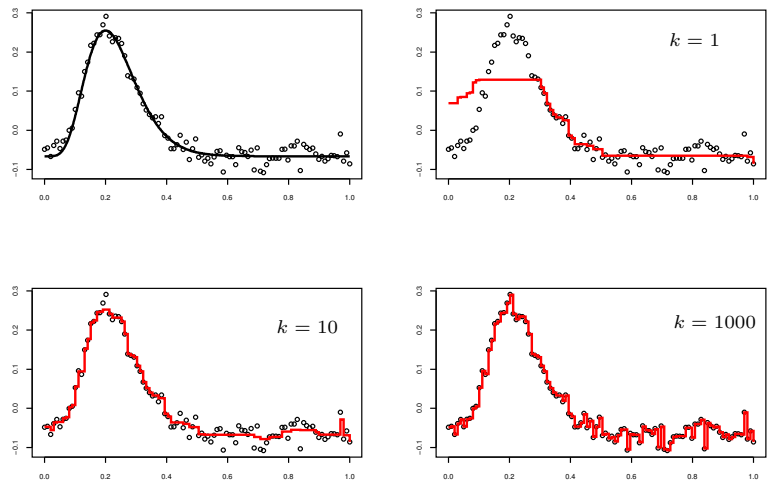}
\end{center}
\caption{Application of the \textsf{I.I.R.} algorithm for $k=1,10$, and $1,000$ iterations.}
\label{fig:iir-interpolation}
\end{figure}
\noindent The bottom right figure illustrates that, as established in Guyader \textit{et al.} \cite{guyader2012}, for fixed sample size $n$, the function $\hat{r}_n^{(k)}(x)$ converges to an interpolant of the data when the number of iterations $k$ tends to infinity, {\it i.e.}, for all $i=1,\dots,n$,
\begin{equation*}
\lim_{k\rightarrow \infty}\hat{r}_n^{(k)}(x_i)=y_i.
\end{equation*}

\noindent One interpretation of the above result is that increasing the number of iterations leads to overfitting. Thus, iterating the procedure until convergence is not desirable. On the other hand, as illustrated on figure \ref{fig:iir-interpolation}, iterations beyond the first step typically improve the fit. This suggests that we need to couple the \textsf{I.I.R.} algorithm with a stopping rule. In this respect, two important remarks are in order. Firstly, since equation (\ref{eq:prolongement}) enables predictions at arbitrary locations $x\in [0,1]$, all the standard data-splitting techniques can be applied to stop the algorithm.\\

\noindent Secondly, the choice of a stopping criterion as a model selection suggests stopping rules based on Akaike Information Criterion, Bayesian Information Criterion or Generalized Cross Validation. These criteria can be written in the generic form   
\begin{equation}\label{eq:stop-commonform}
\argmin_{p}\left\{\log \frac{1}{n}\RSS(p)+\phi(p)\right\}.
\end{equation}\\ 
Here, $\RSS$ denotes the residual sum of squares and $\phi$ is an increasing function. The parameter $p$ stands for the number (or equivalent number) of parameters. For isotonic regression, we refer to Meyer and  Woodroofe \cite{meyer2000} to consider that the number of jumps provides the effective dimension of the model. Therefore, a natural extension for \textsf{I.I.R.} is to replace $p$ by the number of jumps of $\hat{r}_n^{(k)}$ in (\ref{eq:stop-commonform}). The comparisons of these criteria and the practical behavior of the \textsf{I.I.R.} procedure will be addressed elsewhere by the authors.

\section{Isotonic regression: a general result of consistency}\label{sec:isotonic-consistency}
In this section, we focus on the first half step of the algorithm, which consists in applying isotonic regression to the original data. To simplify the notations, we omit in this section the exponent related to the number of iterations $k$, and  simply denote $\hat{u}_n$ the isotonic regression on the data, that is,
\begin{equation*}
\hat{u}_n=\argmin_{u\in{\cal C}^{+}_n}\|y-u\|_n=\argmin_{u\in{\cal C}_n^+}\frac{1}{n}\sum_{i=1}^n\left(y_i-u_i\right)^2.
\end{equation*}
Let $u_+$ denote the closest non-decreasing function to the regression function $r$ with respect to the $L_2(\mu)$ norm. Thus, $u_+$ is defined as 
\begin{equation*}\label{eq:u-plus}
u_+=\argmin_{u\in {\cal C}^+}\|r-u\|=\argmin_{u\in {\cal C}^+}\int_{[0,1]}(r(x)-u(x))^2\mu(dx),
\end{equation*} 
where ${\cal C}^+$ denotes the cone of non-decreasing functions in $L_2(\mu)$. Since ${\cal C}^+$ is closed and convex, the metric projection $u_+$ exists and is unique in $L_2(\mu)$.\\ 

\noindent For mathematical purpose, we also introduce $u_n$, the result from applying isotonic regression to the sample $\left(x_i,r(x_i)\right)$, $i=1,\ldots,n$, that is
\begin{equation}\label{eq:u-n}
u_n=\argmin_{u\in {\cal C}_n^+}\|r-u\|_n=\argmin_{u\in{\cal C}_n^+}\frac{1}{n}\sum_{i=1}^n\left(r(x_i)-u_i\right)^2.
\end{equation}
Finally, we note that, since $r$ is bounded, so are $u_+$ and $u_n$, independently of the sample size $n$ (see for example Lemma 2 in Anevski and Soulier \cite{anevski2011monotone}). Figure \ref{fig:cons-iso} displays the three terms involved.
\begin{figure}[H]
\begin{center}
\input{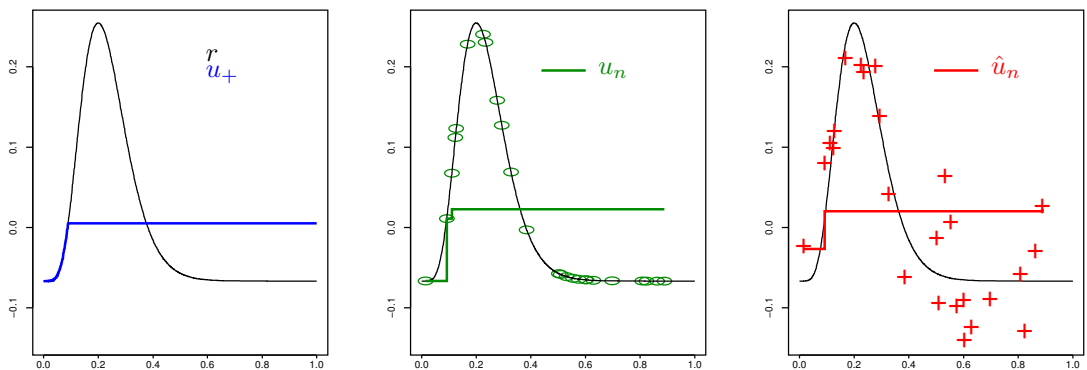}
\end{center}
\caption{Isotonic regression on a non-monotone regression function.}
\label{fig:cons-iso}
\end{figure}
\noindent The main result of this section states that 
\begin{equation*}
\E\left[\|\hat{u}_n-u_+\|^2\right] \underset{n\rightarrow \infty}{\longrightarrow}0,
\end{equation*}
where the expectation is taken with respect to the sample ${\cal D}_n$. Our analysis decomposes $\|\hat{u}_n-u_+\|$ into two distinct terms:
\begin{equation*}\label{eq:bv-theorem}
\|\hat{u}_n-u_+\|\leq \|\hat{u}_n-u_n\|+\|u_n-u_+\|.
\end{equation*}
As $\|u_n-u_+\|$ does not depend on the response variable $Y_i$, one could interpret it as a bias term, whereas $\|\hat{u}_n-u_n\|$ plays the role of a variance term.\\

\noindent Throughout this section, our results are stated for both the empirical norm $\|.\|_n$ and the $L_2(\mu)$ norm $\|.\|$, as both are informative. The following proposition states the convergence of the bias term (its proof is postponed to Section \ref{sec:biais-iso}).
\begin{proposition}\label{lem:biais-iso}
With the previous notations, we have 
$$\lim_{n\rightarrow \infty}\|u_n-u_+\|_n=0 \qquad a.s.,$$
and 
$$\lim_{n\rightarrow \infty}\|u_n-u_+\|=0 \qquad a.s.$$
\end{proposition}

\noindent Applying Lebesgue's dominated convergence Theorem ensures that both
$$\lim_{n\rightarrow \infty}\E\left[\|u_n-u_+\|_n^2\right]=0\qquad \textrm{and} \qquad\lim_{n\rightarrow \infty}\E\left[\|u_n-u_+\|^2\right]=0.$$
Analysis of the variance term requires that we assume that the noise $\varepsilon$ is bounded. It then follows from Anevski and Soulier \cite{anevski2011monotone} that $\hat{u}_n$ is bounded, independently of the sample size $n$. The proof of the following result is given in Section \ref{sec:variance-iso}). 
\begin{proposition}\label{lem:variance-iso}
Assume that the random variable $\varepsilon$ is bounded, then we have
$$\lim_{n\rightarrow \infty}\E\left[\|\hat{u}_n-u_n\|_n^2\right]=0,$$
and 
$$\lim_{n\rightarrow \infty}\E\left[\|\hat{u}_n-u_n\|^2\right]=0.$$
\end{proposition} 
Combining Proposition \ref{lem:biais-iso} and Proposition \ref{lem:variance-iso} yields the following theorem.
\begin{theorem}\label{thm:iso-consistency}
Consider the model $Y=r(X)+\varepsilon$, where $r:[0,1]\rightarrow\R$ belongs to $L_2(\mu)$, $\mu$ is a non-atomic distribution on $[0,1]$,  and $\varepsilon$ is a bounded random variable satisfying $\E\left[\varepsilon|X\right]=0$. Denote $u_+$ and $\hat{u}_n$ the functions resulting from the isotonic regression applied on $r$ and on the sample ${\cal D}_n$, respectively. Then we have
$$\E\left[\|\hat{u}_n-u_+\|_n^2\right]\to 0$$
and
$$\E\left[\|\hat{u}_n-u_+\|^2\right]\to 0$$
when the sample size $n$ tends to infinity.
\end{theorem} 

\noindent This result generalizes the consistency of isotonic regression when applied in a more general context than the one of monotone functions. It will be of constant use when iterating our algorithm. This is the topic of the upcoming section.
 
\section{Consistency of iterative isotonic regression}\label{sec:iir-consistency}
We now proceed with our main result, which states that there is a sequence of iterations $k_n$, increasing with the sample size $n$, such that
$$  \E\left[\|\hat{r}_n^{(k_n)}-r\|^2\right]\underset{n\rightarrow \infty}{\longrightarrow}0.$$
In order to control the expectation of the $L_2$ distance between the estimator $\hat{r}_n^{(k)}$ and the regression function $r$, we shall split $\|\hat{r}_n^{(k)}-r\|$ as follows: let $r^{(k)}$ be the result from applying the algorithm on the regression function $r$ itself $k$ times, that is $r^{(k)}=u^{(k)}+b^{(k)}$, where 
$$ u^{(k)}=\argmin_{u\in {\cal C}^+}\|r-b^{(k-1)}-u\|\qquad \textrm{and} \qquad b^{(k)}=\argmin_{b\in{\cal C}^-}\|r-u^{(k)}-b\|.$$ 
We then upper-bound
\begin{equation}\label{eq:estim-approx}
\|\hat{r}_n^{(k)}-r\| \leq \|r^{(k)}-r\|+\|\hat{r}_n^{(k)}-r^{(k)}\|.
\end{equation}
In this decomposition, the first term is an approximation error, while the second one corresponds to an estimation error.\\ 

\noindent Figure \ref{fig:demarche-consistance} displays the function $r^{(k)}$ for two particular values of $k$. One can see that, after $k$ steps of the algorithm, there generally remains an approximation error $\|r^{(k)}-r\|$. Nonetheless, one also observes that this error decreases when iterating the algorithm.   
\begin{figure}[H]
\begin{center}
\input{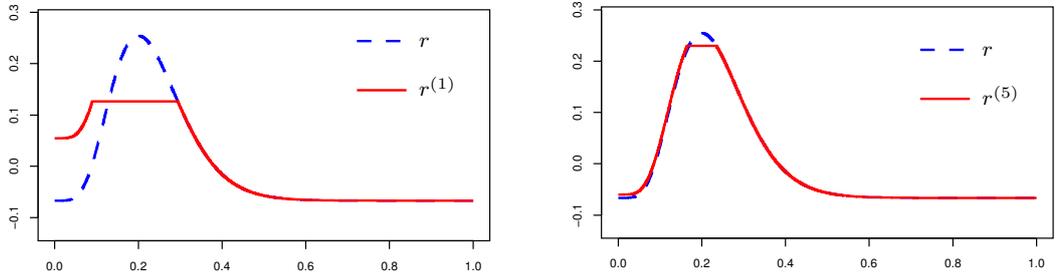}
\end{center}
\caption{Decreasing of the approximation error $\|r^{(k)}-r\|$ with $k$.}
\label{fig:demarche-consistance}
\end{figure}
\noindent The following proposition states that the approximation error can indeed be controlled  by increasing the number of iterations $k$. Its proof relies on the interpretation of \textsf{I.I.R.} as a Von Neumann's algorithm (see Section \ref{sec:approximation} for the proof).
 \begin{proposition}\label{pro:approximation}
Assume that $r$ is a right-continuous function of bounded variation and $\mu$ a non-atomic law on $[0,1]$. Then the approximation term $\|r^{(k)}-r\|$ tends to $0$ when the number of iterations grows:
$$\lim_{k\rightarrow \infty}\|r^{(k)}-r\|=0,$$
where $\|.\|$ denotes the quadratic norm in $L_2(\mu)$. 
\end{proposition}
\noindent Coming back to  (\ref{eq:estim-approx}), we further decompose the estimation error into a bias and a variance term to obtain 
\begin{center}
\begin{tabular}{ccccc}
$\|\hat{r}_n^{(k)}-r\|$ &$\leq$ & $\underbrace{\|\hat{r}_n^{(k)}-r^{(k)}\|}$&$+$&$\underbrace{\|r^{(k)}-r\|}.$\\
                      &        & Estimation                            &    &Approximation \\
                      &        &  $\leq$ &                                &\\ 
                      &  &$\overbrace{\begin{array}{ccc}                 
                  \|\hat{r}_n^{(k)}-r_n^{(k)}\|&+& \|r_n^{(k)}-r^{(k)}\|\\
                  \downarrow & &  \downarrow\\                   
                  \textrm{Variance}& +& \textrm{Bias}
                  \end{array}}$ & &

\end{tabular}
\end{center}
The function $r_n^{(k)}$ results from $k$ iterations of the algorithm on the sample $(x_i,r(x_i))$, $i=1,\ldots,n$, and can be seen as the equivalent of the function $u_n$ defined in (\ref{eq:u-n}). This decomposition allows us to make use of the consistency results of the previous section, and to control the estimation error when the sample size $n$ goes to infinity. We now state the main theorem of this paper.   
\begin{theorem}\label{thm:iir-consistency}
Consider the model $Y=r(X)+\varepsilon$, where $r:[0,1]\rightarrow\R$ is a right-continuous function of bounded variation, $\mu$ a non-atomic distribution on $[0,1]$, and $\varepsilon$ a bounded random variable satisfying $\E\left[\varepsilon|X\right]=0$. Then there exists an increasing sequence of iterations $(k_n)$ such that 
$$  \E\left[\|\hat{r}_n^{(k_n)}-r\|^2\right]\underset{n\rightarrow \infty}{\longrightarrow}0,$$
where $\|.\|$ denotes the quadratic norm in $L_2(\mu)$.  
\end{theorem}
\begin{proof}
\noindent Coming back to the original notation, Theorem \ref{thm:iso-consistency} states that
\begin{equation}\label{eq:bvn-finalstep05}
\lim_{n\rightarrow \infty}\E\left[\|\hat{u}_n^{(1)}-u^{(1)}\|^2_n\right]=0\qquad \textrm{and} \qquad\lim_{n\rightarrow \infty}\E\left[\|\hat{u}_n^{(1)}-u^{(1)}\|^2\right]=0.
\end{equation}
\noindent In the following, we show that this result still holds when applying the backfitting algorithm. Before proceeding, just remark that, since $r$ and $\varepsilon$ are bounded, this will also be the case for all the quantities at stake in the remainder of the proof. In particular, this allows us to use the concentration inequalities established in Section \ref{sec:inegalites-concentration}. \\

\noindent $\bullet$ We first describe the end of the first step by showing that  $\E\left[\|\hat{b}_n^{(1)}-b^{(1)}\|^2\right]\rightarrow 0$.\\
Recall the definitions 
$$ b^{(1)}=\argmin_{b\in {\cal C}^-}\|r-u^{(1)}-b\| \qquad \textrm{and} \qquad \hat{b}_n^{(1)}=\argmin_{b\in{\cal C }^-_n}\|y-\hat{u}_n^{(1)}-b\|_n.$$ 
In order to mimic the previous step, let us consider the vectors
$$
\tilde{y}=y-u^{(1)} \qquad \textrm{and} \qquad \tilde{b}_n^{(1)}=\argmin_{b\in {\cal C}^-_n}\|\tilde{y}-b\|_n,
$$
so that
$$ \tilde{y}=\left(r-u^{(1)}\right)+\varepsilon$$
and 
$$\tilde{b}_n^{(1)}=\argmin_{b\in{\cal C}^-_n}\|(r-u^{(1)})+\varepsilon-b\|_n.$$
 To study the term $\|\tilde{b}_n^{(1)}-b^{(1)}\|$, one can apply {\it mutatis mutandis} the result of Theorem \ref{thm:iso-consistency}, replacing $\hat{u}_n^{(1)}$ by $\tilde{b}_n^{(1)}$, $r$ by $r-u^{(1)}$, and isotonic regression by antitonic regression. Hence,
\begin{equation}\label{oebv}
\lim_{n\rightarrow \infty}\E\left[\|\tilde{b}_n^{(1)}-b^{(1)}\|^2_n\right]=0\qquad \textrm{and} \qquad\lim_{n\rightarrow \infty}\E\left[\|\tilde{b}_n^{(1)}-b^{(1)}\|^2\right]=0.
\end{equation}
As projection reduces distances, we also have
$$\|\hat{b}_n^{(1)}-\tilde{b}_n^{(1)}\|_n\leq \|y-\hat{u}_n^{(1)}-\tilde{y}\|_n=\|\hat{u}_n^{(1)}-u^{(1)}\|_n.$$
Thanks to equations (\ref{eq:bvn-finalstep05}) and (\ref{oebv}), we deduce  
$$\E\left[\|\hat{b}_n^{(1)}-b^{(1)}\|^2_n\right]\leq 2\times
\left\{\E\left[ 
\|\hat{b}_n^{(1)}-\tilde{b}_n^{(1)}\|^2_n\right]+\E\left[\|\tilde{b}_n^{(1)}-b^{(1)}\|^2_n
\right]\right\}\rightarrow 0.$$
Invoking the same arguments as those at the end of the proof of Proposition \ref{lem:variance-iso}, we also have
$$\lim_{n\rightarrow\infty}\E\left[\|\hat{b}_n^{(1)}-b^{(1)}\|^2\right]=0.$$
Finally, at the end of the first iteration, we have  
$$ \E\left[\|\hat{r}_n^{(1)}-r^{(1)}\|^2\right]\leq 2\times\left\{\E\left[\|\hat{u}_n^{(1)}-u^{(1)}\|^2\right]+\E\left[\|\hat{b}_n^{(1)}-b^{(1)}\|^2\right]\right\}\rightarrow 0.$$ 
$\bullet$ For the beginning of the second iteration, consider this time 
$$\hat{u}_n^{(2)}=\argmin_{u\in{\cal C}^+_n}\|y-\hat{b}_n^{(1)}-u\|_n \qquad \textrm{and} \qquad u^{(2)}=\argmin_{u\in{\cal C}^+}\|r-b^{(1)}-u\|.$$
Let us introduce 
 $$ \tilde{y}=y-b^{(1)}=(r-b^{(1)})+\varepsilon \qquad \textrm{and} \qquad \tilde{u}_n^{(2)}=\argmin_{u\in{\cal C}^+_n}\|\tilde{y}-u\|_n=\argmin_{u\in{\cal C}^+_n}\|(r-b^{(1)})+\varepsilon-u\|_n.$$
We apply Theorem \ref{thm:iso-consistency} again, replacing  $r$ by $r-b^{(1)}$, and $\hat{u}_n^{(1)}$ by $\tilde{u}_n^{(2)}$. This leads to  
$$\lim_{n\rightarrow 0}\E\left[\|\tilde{u}_n^{(2)}-u^{(2)}\|^2_n\right]=0.$$
Thanks to the reduction property of isotonic regression and using the conclusion of the first iteration, we get
$$ \E\left[\|\hat{u}_n^{(2)}-\tilde{u}_n^{(2)}\|^2_n\right]\leq \E\left[\|y-\hat{b}_n^{(1)}-((r-b^{(1)})+\varepsilon)\|^2_n\right]= \E\left[\|\hat{b}_n^{(1)}-b^{(1)}\|^2_n\right]\rightarrow 0.$$
Therefore  
$$\E\left[\|\hat{u}_n^{(2)}-u^{(2)}\|^2_n\right]\leq 2\times \left\{\E\left[ \|\hat{u}_n^{(2)}-\tilde{u}_n^{(2)}\|^2_n\right]+\E\left[\|\tilde{u}_n^{(2)}-u^{(2)}\|^2_n\right]\right\}\rightarrow 0$$
and, as before, we also have 
$$\lim_{n\rightarrow\infty}\E\left[\|\hat{u}_n^{(2)}-u^{(2)}\|^2\right]=0.$$
The same scheme leads to $\lim_{n\rightarrow\infty}\E\left[\|\hat{b}_n^{(2)}-b^{(2)}\|^2\right]=0$, so that 
$$ \E\left[\|\hat{r}_n^{(2)}-r^{(2)}\|^2\right]\leq 2\times\left\{\E\left[\|\hat{u}_n^{(2)}-u^{(2)}\|^2\right]+\E\left[\|\hat{b}_n^{(2)}-b^{(2)}\|^2\right]\right\}\rightarrow 0.$$
 $\bullet$ By iterating this process, it is readily seen that, for all $k\geq 1$,
 $$\lim_{n\rightarrow \infty}\E\left[\|\hat{r}_n^{(k)}-r^{(k)}\|^2\right]=0,$$
which means that, at each iteration, the estimation error goes to 0 when the sample size tends to infinity.\\ 

\noindent We deduce that we can construct an increasing sequence $(n_k)$ such that for each $k\geq 1$ and for all $n\geq n_k$ 
$$ \E\left[\|\hat{r}_n^{(k)}-r^{(k)}\|\right]\leq \|r^{(k)}-r\|+\frac{1}{k}.$$
Notice that the term $\|r^{(k)}-r\|$ might be equal to zero ({\it e.g.}, $r^{(1)}=r$ if $r$ is monotone), hence the additive term $1/k$ in the previous inequality. Consequently,
$$ \E\left[\|\hat{r}_n^{(k)}-r\|\right]\leq 2\|r^{(k)}-r\|+\frac{1}{k}.$$
Then let us consider the sequence $(k_n)$ defined as:  $k_n=0$ if $n< n_1$,  $k_n=1$ if $n_1\leq n< n_2$, and so on.
Obviously $(k_n)$ tends to infinity and
 $$\E\left[\|\hat{r}_n^{(k_n)}-r\|\right]\leq2\|r^{(k_n)}-r\|+\frac{1}{k_n}\xrightarrow[n\to\infty]{}0.$$
This ends the proof of Theorem \ref{thm:iir-consistency}.
\end{proof}

\section{Proofs}\label{sec:appendix}
\subsection{Proof of Proposition \ref{lem:biais-iso}}\label{sec:biais-iso}
For $g$ and $h$ two functions defined on $[0,1]$, we denote $\Delta_n(g-h)$ the random variable
$$
\Delta_n(g-h)=\|g-h\|_n^2-\|g-h\|^2=\frac{1}{n}\sum_{i=1}^n\left\{(g(X_i)-h(X_i))^2-\E\left[(g(X)-h(X))^2\right]\right\}.
$$
We first show that
\begin{equation}\label{eq:biais1}
\|r-u_n\|_n\rightarrow \|r-u_{+}\| \qquad a.s.
\end{equation}
To this end, we proceed in two steps, proving in a first time that 
\begin{equation}\label{eq:limsup1} 
 \limsup \|r-u_n\|_n\leq \|r-u_{+}\| \qquad a.s.
\end{equation}
and  in a second time that
\begin{equation}\label{eq:limsup2}
\liminf \|r-u_n\|_n\geq \|r-u_{+}\| \qquad a.s.
\end{equation} 
For the first inequality, let us denote
$$A_n=\left\{|\Delta_n(r-u_{+})|>n^{-1/3}\right\}=\left\{|\|r-u_{+}\|_n^2-\|r-u_{+}\|^2|>n^{-1/3}\right\}.$$
By the definition of $u_{n}$, note that for all $n$,
$$\|r-u_n\|_n\leq \|r-u_{+}\|_n$$
so that on $\overline{A_n}$,
$$\|r-u_n\|_n^2\leq \|r-u_{+}\|_n^2\leq \|r-u_{+}\|^2+n^{-1/3}.$$
Consequently
$$ B_n=\left\{\|r-u_n\|^2_n\leq \|r-u_{+}\|^2+n^{-1/3}\right\}\supset \overline{A_n}.$$
Therefore
$$\P\left(\liminf B_n\right)\geq \P\left(\liminf\overline{A_n}\right)=1-\P\left(\limsup A_n\right).$$
Invoking Lemma \ref{lem:annexe-concentration1} and Borel-Cantelli Lemma, we conclude that $\P\left(\limsup A_n\right)=0$, and hence $\P\left(\liminf B_n \right)=1.$ On the set $\liminf B_n$, we have 
$$ \limsup \|r-u_n\|_n^2\leq \|r-u_{+}\|^2,$$
which proves Equation (\ref{eq:limsup1}).\\

\noindent Conversely, we now establish Equation (\ref{eq:limsup2}). By definition of $u_{+}$, observe that for all $n$,
$$ \qquad \|r-u_{+}\|\leq \|r-u_n\|.$$
Consider the sets
$$C_n=\left\{ \sup_{h\in{\cal C}_{[a,b]}^+}\vert\Delta_n(r-h)\vert>n^{-1/3}\right\}\ \textrm{and}\ D_n=\left\{\|r-u_n\|_n^2\geq \|r-u_{+}\|^2-n^{-1/3}\right\}$$
so that $\overline{C_n} \subset D_n$, and by applying   Lemma \ref{lem:annexe-concentration2},
$$\P\left(\liminf D_n\right) \geq 1-\P\left(\limsup C_n\right)=1.$$ 
On the set $\liminf D_n$, one has 
$$\liminf\|r-u_n\|_n^2\geq\|r-u_+\|^2,$$
which proves (\ref{eq:limsup2}). Combining Equations (\ref{eq:limsup1}) and (\ref{eq:limsup2}) leads to (\ref{eq:biais1}).\\ 
 
%$$\|r-u_n\|_n^2\rightarrow \|r-u_{+}\|^2 \qquad a.s.$$
\noindent Next, using Lemma \ref{lem:annexe-concentration2} again, we get
\begin{equation*}\label{eq:res1-1}
\lim_{n\rightarrow \infty}\|r-u_n\|_n-\|r-u_n\|=0 \qquad a.s.
\end{equation*} 
Combined with (\ref{eq:biais1}), this leads to
\begin{equation}\label{eq:res1-2}
\|r-u_n\|\rightarrow \|r-u_{+}\| \qquad a.s.
\end{equation}
It remains to prove the almost sure convergence of $u_n$ to $u_+$. For this, it suffices to use the parallelogram law. Indeed, noting $m_n=(u_n+u_{+})/2$, we have  
$$\|u_n-u_{+}\|^2=2\left(\|r-u_{+}\|^2+\|u_n-r\|^2\right)-4\|m_n-r\|^2.$$
Since both $u_{+}$ and $u_n$ belong to the convex set ${\cal C}^+$, so does $m_n$. Hence $\|r-u_{+}\|^2\leq \|r-m_n\|^2$, and 
$$ \|u_n-u_{+}\|^2\leq 2\left(\|u_n-r\|^2-\|r-u_{+}\|^2\right).$$
Combining this with (\ref{eq:res1-2}), we conclude that
\begin{equation*}\label{eq:final-biais}
\lim_{n\rightarrow \infty}\|u_n-u_{+}\|=0 \qquad a.s.
\end{equation*}
Finally, Lemma \ref{lem:annexe-concentration2} guarantees the same result for the empirical norm, that is 
\begin{equation*}\label{eq:final-biaisbis}
\lim_{n\rightarrow \infty}\|u_n-u_{+}\|_n=0 \qquad a.s.
\end{equation*}
and the proof is complete.

\subsection{Proof of Proposition \ref{lem:variance-iso}}\label{sec:variance-iso}

Let us denote $\langle \cdot , \cdot \rangle_n$ the inner product associated to the empirical norm $\|.\|_n$. Since isotonic regression corresponds to the metric projection onto the closed convex cone ${\cal C}^+_n$ with respect to this empirical norm, the vectors $\hat{u}_n$ et $u_n$ are characterized by the following inequalities: for any vector $u \in {\cal C}^+_n$,
\begin{eqnarray}
&&\langle y-\hat{u}_n ,  u - \hat{u}_n\rangle_n \, \leq 0 \label{eq:1}\\
&& \langle r - u_n , u - u_n\rangle_n \, \leq 0 \label{eq:2}
\end{eqnarray} 
Setting $u=u_n$ in (\ref{eq:1}) and $u=\hat{u}_n$ in (\ref{eq:2}), we get
\[
\langle y -\hat{u}_n,u_n -\hat{u}_n\rangle_n \; \leq 0  \qquad \textrm{ and } \qquad  \langle r-u_n , \hat{u}_n - u_n\rangle_n \; \leq 0.
\]
Since $\varepsilon=y-r$, this leads to
\begin{equation} \label{eq:3}
\| \hat{u}_n - u_n \|_n^2 \,\, \leq \,\, \langle \varepsilon ,  \hat{u}_n - u_n \rangle_n.
\end{equation}

\noindent Next, we have to use an approximation result, namely Lemma \ref{lem:nick} in Section \ref{lknxn}. The underlying idea is to exploit the fact that any non-decreasing bounded sequence can be approached by the element of a subspace  $H_+$ at distance less than $\delta$. Specifically, if $C$ is an upper-bound for the absolute value of the considered non-decreasing bounded sequences, we can construct such a subspace $H_+$ with dimension $N$ where $N=(8C^2)/\delta^2$. From now on, we will take $N\leq n$. Before proceeding, just notice that the boundedness assumption on the random variables $\varepsilon_i$ allows us to find a common upper bound $C$ for the absolute values of the components of $\hat{u}_n$ and $u_n$.\\

\noindent Let us introduce the vectors $\hat{h}_n$ and $h_n$ defined by
$$ \hat{h}_n=\inf_{h\in H_+}\|\hat{u}_n-h\|_n   \qquad \textrm{ and } \qquad   h_n=\inf_{h\in H_+}\|u_n-h\|_n$$
so that
$$ \|\hat{u}_n-\hat{h}_n\|_n\leq \delta \qquad \textrm{ and } \qquad \|u_n-h_n\|_n\leq \delta.$$
From this, we get
\begin{align*}
\langle \varepsilon,\hat{u}_n-u_n\rangle_n=&\langle \varepsilon,\hat{u}_n-\hat{h}_n\rangle_n+\langle \varepsilon,\hat{h}_n-h_n\rangle_n+\langle \varepsilon,h_n-u_n\rangle_n\\
\leq& \|\hat{h}_n-h_n\|_n\left < \varepsilon,\frac{\hat{h}_n-h_n}{\|\hat{h}_n-h_n\|_n}\right >_n+2\delta\|\varepsilon\|_n\\
\leq& \left\{\|\hat{h}_n-\hat{u}_n\|_n+\|\hat{u}_n-u_n\|_n+\|u_n-h_n\|_n\right\}\sup_{v\in H_+,\|v\|_n=1}\langle \varepsilon,v\rangle_n+2\delta\|\varepsilon\|_n\\
\leq& \left\{\|\hat{u}_n-u_n\|_n+2\delta\right\}\sup_{v\in H_+,\|v\|_n=1}\langle \varepsilon,v\rangle_n+2\delta\|\varepsilon\|_n.
\end{align*}  
According to (\ref{eq:3}), we deduce
$$
\| \hat{u}_n - u_n \|_n^2\leq \left\{\|\hat{u}_n-u_n\|_n+2\delta\right\}\sup_{v\in H_+,\|v\|_n=1}\langle \varepsilon,v\rangle_n+2\delta\|\varepsilon\|_n$$
so that
\begin{equation*}
%\label{eq:rappel}
\|\hat{u}_n-u_n\|_n^2\leq \left\{\|\hat{u}_n-u_n\|_n+2\delta\right\}\|\pi_{H_+}(\varepsilon)\|_n+2\delta\|\varepsilon\|_n,
\end{equation*}
where $\pi_{H_+}(\varepsilon)$ stands for the metric projection of $\varepsilon$ onto $H_+$. Put differently, we have
$$\|\hat{u}_n-u_n\|_n^2\leq \|\hat{u}_n-u_n\|_n\times\|\pi_{H_+}(\varepsilon)\|_n+2\delta\left\{\|\pi_{H_+}(\varepsilon)\|_n+\|\varepsilon\|_n\right\},$$ 
and taking the expectation on both sides leads to
$$ \E\left[\|\hat{u}_n-u_n\|_n^2\right]\leq \E\left[\|\hat{u}_n-u_n\|_n\times\|\pi_{H_+}(\varepsilon)\|_n\right]+2\delta\left\{\E\left[\|\pi_{H_+}(\varepsilon)\|_n\right]+\E\left[\|\varepsilon\|_n\right]\right\}.$$
If we denote
$$\left\{ \begin{array}{rl}
x^2&=\E\left[\|\hat{u}_n-u_n\|^2_n\right]\\
\alpha_n&=\sqrt{\E\left[\|\pi_{H_+}(\varepsilon)\|_n^2\right]}\\
\beta_n&=2\delta\left\{\E\left[\|\pi_{H_+}(\varepsilon)\|_n\right]+\E\left[\|\varepsilon\|_n\right]\right\}
\end {array}\right .$$
an application of Cauchy-Schwarz inequality gives
$$x^2-\alpha_nx-\beta_n\leq 0\ \Rightarrow\ x\leq \frac{\alpha_n+\sqrt{\alpha_n^2+4\beta_n}}{2},$$
which means that  
$$ \E\left[\|\hat{u}_n-u_n\|^2_n\right]\leq \left(\frac{\alpha_n+\sqrt{\alpha_n^2+4\beta_n}}{2}\right)^2.$$
Since the random variables $\varepsilon_i$ are i.i.d. with mean zero and common variance $\sigma^2$, a straightforward computation shows that
$$\E\left[\|\pi_{H_+}(\varepsilon)\|_n^2\right]=\frac{1}{n}\E\left[(\pi_{H_+}\varepsilon)'(\pi_{H_+}\varepsilon)\right]=\frac{1}{n}\E\left[\trace\left((\pi_{H_+}\varepsilon)'(\pi_{H_+}\varepsilon)\right)\right]=\frac{1}{n}\trace\left(\E\left[\varepsilon\varepsilon'\right]\pi_{H_+}\right),$$
and since $H_+$ has dimension $N=(8C^2)/\delta^2$, this gives
$$\E\left[\|\pi_{H_+}(\varepsilon)\|_n^2\right]=\sigma^2\frac{N}{n}\ \Rightarrow\  \alpha_n=\sigma \sqrt{\frac{N}{n}}=\frac{2\sqrt{2}C\sigma}{\delta\sqrt{n}}.$$
Set $\delta=\delta_n=n^{-\alpha}$ with $0<\alpha<1/2$, it then follows that $\alpha_n$ goes to zero when $n$ goes to infinity. Moreover, Jensen's inequality implies 
$$ \beta_n\leq 2\delta(\alpha_n+\sigma).$$
As both $\delta =\delta_n$ and $\alpha_n$ tend to zero when $n$ goes to infinity, we have shown that 
\begin{equation}\label{eq:part1}
\lim_{n\rightarrow \infty}\E\left[\|\hat{u}_n-u_n\|_n^2\right]=0.
\end{equation}
Remark that for any non negative random variable $X$,
$$\E[X]=\int_{0}^{+\infty}\P\left(X\geq t\right)dt\leq n^{-1/4}+\int_{0}^{+\infty}\P\left(X\geq t\right)\un_{\{t\geq n^{-1/4}\}}\ dt.$$
From equation (\ref{ieuzoc}) in the proof of Lemma \ref{lem:annexe-concentration3}, we know that for any $t>0$,
$$\P\left(\left|\|\hat{u}_n-u_n\|_n^2-\|\hat{u}_n-u_n\|^2\right|\geq t\right)\leq\exp\left(2\left\lceil\frac{64C^2}{t}\right\rceil\log n-\frac{t^2n}{32C^2}\right).$$
Thus, setting
$$f_n(t)=\un_{[0,n^{-1/4}]}(t)+\exp\left(2\left\lceil\frac{64C^2}{t}\right\rceil\log n-\frac{t^2n}{32C^2}\right)\un_{\{t\geq n^{-1/4}\}},$$
we deduce that
$$\E\left[\left|\|\hat{u}_n-u_n\|_n^2-\|\hat{u}_n-u_n\|^2\right|\right]\leq n^{-1/4}+\int_{0}^{+\infty}f_n(t)\ dt.$$
Then, it is readily seen that there exists an integer $n_0$ such that for all $n\geq n_0$ and for all $t\geq 0$, one has $f_n(t)\leq f_2(t)$. Since for all $t>0$ fixed, $f_n(t)$ goes to 0 when $n$ tends to infinity, it remains to invoke Lebesgue's dominated convergence Theorem to conclude
$$\E\left[\|\hat{u}_n-u_n\|_n^2\right]-\E\left[\|\hat{u}_n-u_n\|^2\right]\rightarrow 0.$$ 
Combining the latter with equation (\ref{eq:part1}), we have obtained
\begin{equation*}%\label{eq:final1}
\lim_{n\rightarrow \infty}\E\left[\|\hat{u}_n-u_n\|\right]^2=0,
\end{equation*}
which ends the proof of Proposition \ref{lem:variance-iso}.

\subsection{Proof of Proposition \ref{pro:approximation}}\label{sec:approximation}
Consider the translated cone
$$r+{\cal C}^+=\{r+u,u\in{\cal C}^+\}.$$
Figure \ref{fig:von-neumann1} provides a very simple interpretation of the algorithm:  namely, it illustrates that the sequences of functions $u^{(k)}$ and $r-b^{(k)}$ might be seen as alternate projections onto the cones ${\cal C}^+$ and $r+{\cal C}^+$. In what follows, we justify this illuminating geometric interpretation in a rigorous way, and we explain its key role in the proof of the convergence as $k$ goes to infinity.\\

\noindent By definition, we have $u^{(1)}=P_{{\cal C}^+}(r)$ where $P_{{\cal C}^+}$ denotes the metric projection onto ${\cal C}^+$. Classical properties of projections ensure that 
$$P_{r+{\cal C}^+}(u^{(1)})=r+P_{{\cal C}^+}(u^{(1)}-r)=r-P_{{\cal C}^-}(r-u^{(1)}).$$
Coming back to the  definition of  $b^{(1)}=P_{{\cal C}^-}(r-u^{(1)})$, we are led to
$$r-b^{(1)}=P_{r+{\cal C}^+}(u^{(1)}).$$
In the same manner, since $u^{(2)}=P_{{\cal C}^+}(r-b^{(1)})$, we get 
$$r-b^{(2)}=r-P_{{\cal C}^-}(r-u^{(2)})=r+P_{{\cal C}^+}(r-u^{(2)})=P_{r+{\cal C}^+}(u^{(2)}).$$
More generally, denoting $b^{(0)}=0$, this yields for all $k\geq 1$ (see also figure \ref{fig:von-neumann1})
\begin{equation*}%\label{eq:vn-iir}
u^{(k)}=P_{{\cal C}^+}(r-b^{(k-1)}) \qquad \textrm{and} \qquad r-b^{(k)}=P_{r+{\cal C}^+}(u^{(k)}).
\end{equation*}
\begin{figure}[H]
\begin{center}
\input{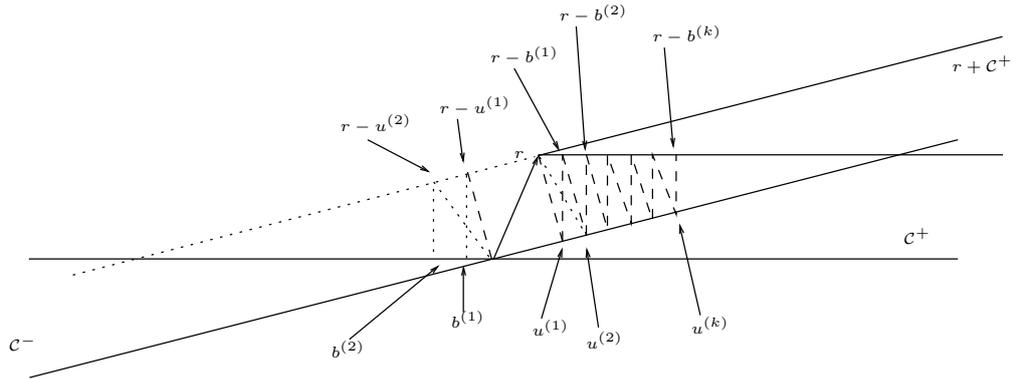}
\end{center}
\caption{Interpretation of \textsf{I.I.R.} as a Von Neumann's algorithm.}
\label{fig:von-neumann1}
\end{figure}
\noindent It remains to invoke Theorem 4.8 in Bauschke and  Borwein \cite{bauschke1994dykstra} to conclude that  $$(r-b^{(k)})-u^{(k)}=r-r^{(k)}\xrightarrow[k\to\infty]{} 0,$$
which ends the proof of Proposition \ref{pro:approximation}.

\section{Technical results }\label{tech}
\subsection{Concentration inequalities}\label{sec:inegalites-concentration}
Throughout the previous proofs, we repeatedly needed to pass from the empirical norm $\|.\|_n$ to the $L_2(\mu)$ norm $\|.\|$. This was made possible thanks to several exponential inequalities that we justify in this section.\\

\noindent Specifically, let $g$ and $h$ denote two mappings from $I=[0,1]$ to $[-C,C]$, and consider the random variable
\begin{equation*}\label{eq:deltadev-annexe}
\Delta_n(g-h)=\frac{1}{n}\sum_{i=1}^n\left\{(g(X_i)-h(X_i))^2-\E\left[(g(X)-h(X))^2\right]\right\}=\|g-h\|_n^2-\|g-h\|^2. 
\end{equation*}
In what follows, we focus on the concentration of $\Delta_n(g-h)$ around zero. The first result is a straightforward application of Hoeffding's inequality. 

\begin{lem}\label{lem:annexe-concentration1}
For any couple of mappings $g$ and $h$ from $[0,1]$ to $[-C,C]$, there exist positive real numbers $\alpha$, $\beta$, $c_1$ and $c_2$, depending only on $C$, and such that
\begin{equation*}\label{eq:lem1-1}
 \P\left(\vert\Delta_n(g-h)\vert>n^{-\alpha}\right)\leq c_1\exp\left(-c_2n^{\beta}\right). 
\end{equation*}
\end{lem}

\noindent \begin{proof}
Since $|g(X_i)-h(X_i)|\leq 2C$, Hoeffding's inequality gives for all $t>0$
\begin{equation}\label{eq:lem1-0}  
\P\left(|\Delta_n(g-h)|>t\right)\leq 2\exp\left(-\frac{t^2n}{8C^2}\right)
\end{equation}
Taking $t=n^{-\alpha}$ with $\alpha\in(0,1/2)$, we deduce
\begin{equation*}\label{eq:lem1-3}
\P\left(|\Delta_n(h)|>n^{-\alpha}\right)\leq 2\exp\left(-\frac{n^{1-2\alpha}}{8C^2}\right)
\end{equation*}
and the result is proved with $c_1=2$, $c_2=1/(8C^2)$ and $\beta=1-2\alpha>0$.
\end{proof}

\noindent The next lemma goes one step further, by considering, for fixed $g$, the tail distribution of 
$$\sup_{h\in {\cal C}^+_{[0,1]}}|\Delta_n(g-h)|.$$
For obvious reasons, this type of result is sometimes called a maximal inequality. The proof shares elements with the one of Theorem 3.1 of van de Geer and Wegkamp \cite{vdgw}.  

 \begin{lem}\label{lem:annexe-concentration2}
Let $g$ be a function from $[0,1]$ to $[-C,C]$ and let ${\cal C}^+_{[0,1]}$ denote the set of non-decreasing functions from $[0,1]$ to $[-C,C]$. There exist positive real numbers $\alpha'$, $\beta'$, $c_1'$ and $c_2'$ depending only on $C$ and such that 
\begin{equation*}\label{eq:lem2-1}
\P\left(\sup_{h\in {\cal C}^+_{[0,1]}}|\Delta_n(g-h)|>n^{-\alpha'}\right)\leq c'_1\exp\left(-c'_2n^{\beta'}\right).
\end{equation*}
\end{lem}
\begin{proof}
The first step consists in showing that the mapping $h\mapsto \Delta_n(g-h)$ is Lipschitz. For any pair of functions $h$ and $\tilde{h}$,  we have 
\begin{align*}
\Delta_n(g-h)-\Delta_n(g-\tilde{h})=&\frac{1}{n}\sum_{i=1}^n\left\{2g(X_i)-h(X_i)-\tilde{h}(X_i)\right\}\left(\tilde{h}(X_i)-h(X_i)\right)\\
&-\E\left[\left\{2g(X)-h(X)-\tilde{h}(X)\right\}\left(\tilde{h}(X)-h(X)\right)\right].
\end{align*}
Since $h$ and $\tilde{h}$ take values in $[-C,C]$, we get
$$
\vert\Delta_n(g-h)-\Delta_n(g-\tilde{h})\vert\leq 4C\times\left\{\frac{1}{n}\sum_{i=1}^n\vert h(X_i)-\tilde{h}(X_i)\vert+\E\left[\vert h(X)-\tilde{h}(X)\vert\right]\right\} 
$$
and according to Jensen's inequality,
$$ \vert\Delta_n(g-h)-\Delta_n(g-\tilde{h})\vert\leq 4C\times \left\{\|h-\tilde{h}\|_n+\|h-\tilde{h}\|\right\}.$$
Now, since $\|h-\tilde{h}\|=\E\left[\|h-\tilde{h}\|_n\right]$, if the inequality $\|h-\tilde{h}\|_n\leq \delta$ is satisfied, we also have $\|h-\tilde{h}\|\leq \delta$.  Thus, 
$$ \forall \delta >0,\qquad\|h-\tilde{h}\|_n\leq \delta\Rightarrow  \vert\Delta_n(g-h)-\Delta_n(g-\tilde{h})\vert\leq 8C\delta$$
and the mapping $h\mapsto \Delta_n(g-h)$ is Lipschitz for the empirical norm $\|\cdot\|_n$.\\

\noindent Next, let us consider a $\delta$-covering ${\cal E}^*=\{e^*_j,j=1,\cdots, M\}$ of ${\cal C}^+_{[0,1]}$ for the empirical norm $\|.\|_n$. We stress that this set ${\cal E}^*$ is random since it depends on the points $X_i$, but its cardinality $M$ may be chosen deterministic and upper-bounded as follows (see Lemma \ref{combinatoire}): denoting $N=\left\lceil\frac{2C}{\delta}\right\rceil$, where $\left\lceil \right\rceil$ stands for the ceiling function, we have
\begin{equation}\label{zecnv} 
M=\binom{n+N}{N}\leq n^N,
\end{equation}
where the last inequality is satisfied for any integer $n\geq 2$ as soon as $N\geq 3$.\\

\noindent Then, for any $h$ in ${\cal C}^+_{[0,1]}$, there exists $e^*$ in ${\cal E}^*$ such that $\|h-e^*\|_n\leq \delta$. From the previous Lipschitz property, we know that
$$  \vert\Delta_n(g-h)-\Delta_n(g-e^*)\vert \leq 8C\delta.$$
Letting $t>0$ and $\delta=t/(16C)$, our objective is to upper bound
$$\P\left(\sup_{h\in{\cal C}^+_{[0,1]}}\vert\Delta_n(g-h)\vert>t\right).$$
In this aim, for any $h$ in ${\cal C}^+_{[0,1]}$ and any $e^*$ in ${\cal E}^*$, we start with the decomposition 
$$\vert\Delta_n(g-h)\vert \leq \vert\Delta_n(g-h)-\Delta_n(g-e^*)\vert+\vert\Delta_n(g-e^*)\vert.$$ 
For any $h$ such that $\vert \Delta_n(g-h)\vert>t$, since there exists $e^*$ in ${\cal E}^*$ such that
$$\vert\Delta_n(g-h)-\Delta_n(g-e^*)\vert\leq t/2,$$
we necessarily have $\vert \Delta_n(g-e^*)\vert >t/2$, and consequently
$$\P\left(\vert\Delta_n(g-h)\vert>t\right)\leq \P\left(\max_{j=1\cdots M}\vert\Delta_n(g-e^*_j)\vert>t/2\right).$$
In other words,
\begin{align*}
\P\left(\sup_{h\in{\cal C}^+_{[0,1]}}\vert\Delta_n(g-h)\vert>t\right)&\leq \P\left(\max_{j=1\cdots M}\vert\Delta_n(g-e^*_j)\vert>t/2\right)\\
&\leq \P\left(\bigcup_{j=1}^M\vert\Delta_n(g-e^*_j)\vert>t/2\right)\\
&\leq \sum_{j=1}^M\P\left(\vert\Delta_n(g-e^*_j)\vert>t/2\right).
\end{align*}
According to (\ref{eq:lem1-0}) and to the fact that 
$$M\leq n^N=n^{\left\lceil\frac{2C}{\delta}\right\rceil},$$
fixing $\delta=t/(16C)$ leads to 
$$ 
\P\left(\sup_{h\in{\cal C}^+_{[0,1]}}\vert\Delta_n(g-h)\vert>t\right)\leq 2M\exp\left(-\frac{t^2n}{8C^2}\right)\leq2\exp\left(\left\lceil\frac{32C^2}{t}\right\rceil\log n-\frac{t^2n}{32C^2}\right).
$$
Finally, for any $\alpha'\in(0,1/3)$, there exists $c'_2=c'_2(\alpha')$ such that for any integer $n$,
$$\left\lceil\frac{32C^2}{n^{-\alpha'}}\right\rceil\log n-\frac{n^{-2\alpha'}n}{32C^2}\leq -c'_2n^{1-2\alpha'},$$
hence the desired result with $t=n^{-\alpha'}$ and $\beta'=1-2\alpha'$. 
\end{proof}

\noindent The last concentration inequality is a generalization of the previous one: this time, neither $g$ nor $h$ are assumed fixed.

\begin{lem}\label{lem:annexe-concentration3}
Let us denote ${\cal C}^+_{[0,1]}$ the set of non decreasing mappings from $[0,1]$ to $[-C,C]$. There exist positive real numbers $\alpha''$, $\beta''$, $c''_1$ and $c''_2$, depending only on $C$, and such that
\begin{equation*}\label{eq:lem3-1}
\P\left(\sup_{h_1\in {\cal C}^+_{[0,1]},h_2\in {\cal C}^+_{[0,1]}}\vert\Delta_n(h_1-h_2)\vert>n^{-\alpha''}\right)\leq c''_1\exp\left(-c''_2n^{\beta''}\right).
\end{equation*}
\end{lem}
\begin{proof}
With the same notations as before, just note that for any mapping $h_1\in{\cal C}^+_{[0,1]}$ (respectively $h_2$), there exists $h^*_1$ (respectively $h^*_2$) in the $\delta$-covering ${\cal E}^*$ of ${\cal C}^+_{[0,1]}$, such that
$$ \|h_1-h_1^*\|_n\leq \delta \qquad\mbox{and}\qquad  \|h_2-h_2^*\|_n\leq \delta.$$
Following the same line as in the proof of the previous lemma, we have, for any mapping $g$ with values in $[-C,C]$, that
$$\vert \Delta_n(g-h_1)-\Delta_n(g-h^*_1)\vert\leq 8C\delta \qquad\mbox{and}\qquad \vert \Delta_n(g-h_2)-\Delta_n(g-h^*_2)\vert\leq 8C\delta.$$
In particular
$$\vert \Delta_n(h_2-h_1)-\Delta_n(h_2-h^*_1)\vert\leq 8C\delta \qquad\mbox{and}\qquad \vert \Delta_n(h^*_1-h_2)-\Delta_n(h^*_1-h^*_2)\vert\leq 8C\delta.$$
Moreover,
$$\vert \Delta_n(h_1-h_2)\vert \leq \vert \Delta_n(h_2-h_1)-\Delta_n(h_2-h^*_1)\vert+\vert\Delta_n(h_2-h^*_1)\vert.$$
Set $\delta=t/(32C)$, then
$$\vert \Delta_n(h_1-h_2)\vert>t \Rightarrow \vert\Delta_n(h_2-h^*_1)\vert>3t/4.$$
In the same manner,
$$\vert \Delta_n(h_2-h^*_1)\vert \leq \vert \Delta_n(h^*_1-h_2)-\Delta_n(h^*_1-h^*_2)\vert+\vert\Delta_n(h^*_1-h^*_2)\vert,$$
and
$$\vert \Delta_n(h_2-h^*_1)\vert>3t/4 \Rightarrow \vert\Delta_n(h^*_1-h^*_2)\vert>t/2.$$
Hence, for any $h_1$ and $h_2$ in ${\cal C}^+_{[0,1]}$, 
$$\P\left(\vert\Delta_n(h_1-h_2)\vert>t\right)\leq \P\left(\max_{h^*_1,h^*_{2}\in{\cal E}^*}\vert\Delta_n(h^*_1-h^*_2)\vert>t/2\right).$$
As a consequence, the choice $\delta=t/(32C)$ gives
\begin{align*}
\P\left(\sup_{h_1\in{\cal C}^+_{[0,1]},h_2\in{\cal C}^+_{[0,1]}}\vert \Delta_n(h_1-h_2)\vert>t\right)&\leq \P\left(\max_{h^*_1,h^*_{2}\in{\cal E}^*}\vert\Delta_n(h^*_j-h^*_{j'})\vert>t/2\right)\\
&\leq \sum_{1\leq j_1\neq j_2\leq M}\P\left(\vert\Delta_n(e^*_j-e^*_{j'})\vert>t/2\right)\\
&\leq M^2\exp\left(-\frac{t^2n}{32C^2}\right).
\end{align*}
According to (\ref{zecnv}), we are led to
\begin{equation}\label{ieuzoc}
\P\left(\sup_{h_1\in{\cal C}^+_{[0,1]},h_2\in{\cal C}^+_{[0,1]}}\vert \Delta_n(h_1-h_2)\vert>t\right)\leq \exp\left(2\left\lceil\frac{64C^2}{t}\right\rceil\log n-\frac{t^2n}{32C^2}\right).
\end{equation}
For any $\alpha''\in(0,1/3)$, there exists a real number $c''_2=c''_2(\alpha'')$ such that for any integer $n$
$$2\left\lceil\frac{64C^2}{n^{-\alpha''}}\right\rceil\log n-\frac{n^{-2\alpha''}n}{32C^2}\leq -c''_2n^{1-2\alpha''},$$
hence the desired result with $t=n^{-\alpha''}$ and $\beta''=1-2\alpha''$. 
\end{proof}

\noindent We conclude this section with the proof of inequality (\ref{zecnv}). It borrows elements from Lemma 3.2 in van de Geer \cite{vdg}. 
 
\begin{lem}\label{combinatoire}
Denote ${\cal C}^+_{[0,1]}$ the set of non-decreasing mappings from $[0,1]$ to $[-C,C]$, and $\|.\|_n$ the empirical norm with respect to the sample $(X_1,\dots,X_n)$. For any $\delta>0$, there exists a $\delta$-covering of $({\cal C}^+_{[0,1]},\|.\|_n)$ with cardinality less than $M=\binom{n+N}{N}$, where $N=\left\lceil\frac{2C}{\delta}\right\rceil$, and $\left\lceil \right\rceil$ stands for the ceiling function.
\end{lem}

\noindent \begin{proof} Let us rewrite $X_{(1)}\leq\dots\leq X_{(n)}$ the reordering of the sample $(X_1,\dots,X_n)$ in increasing order. Recall that the empiric norm is defined for any pair of functions $g$ and $h$ in ${\cal C}^+_{[0,1]}$ by
$$\|g-h\|_n=\sqrt{\frac{1}{n}\sum_{i=1}^n(g(X_{(i)})-h(X_{(i)}))^2},$$
Hence, if $|g(X_{(i)})-h(X_{(i)})|\leq \delta$ for all indices $i=1,\cdots,n$, we also have $\|g-h\|_n\leq\delta$.\\

\noindent For the sake of simplicity, let us assume that $N_0=C/\delta$ is an integer and let us consider the following partition of the interval $[-C,C]$
$${\cal S}=\left\{-C=-N_0\delta<-(N_0-1)\delta<\dots<-\delta<0<\delta<\dots<(N_0-1)\delta<N_0\delta=C\right\}.$$   
Let us denote ${\cal I}^+_{[0,1]}$ the set of non-decreasing functions defined on $[0,1]$, with values in ${\cal S}$ and piecewise constant on the intervals $(X_{(i)},X_{(i+1)})$. We also suppose that they are constant on the intervals $[0,X_{(1)}]$ and $[X_{(n)},1]$,  with respective values the ones of $X_{(1)}$ and $X_{(n)}$.\\

\noindent Firstly, it is readily seen that any function $g$ in ${\cal C}^+_{[0,1]}$ may be approximated at a distance less than or equal to $\delta$ with respect to the empirical norm $\|.\|_n$ by a function in ${\cal I}^+_{[0,1]}$ . For this, it indeed suffices to pick at each point $X_{(i)}$ the nearest value of $g(X_{(i)})$ in the partition ${\cal S}$. Secondly, it is well-known in discrete mathematics (see for example Lov{\'a}sz {\it et al.} \cite{lkv}, Theorem 3.4.2) that 
$$\vert{\cal I}^+_{[0,1]}\vert=\binom{n+N}{N}.$$
\end{proof}

\subsection{Proof of Lemma \ref{lem:nick}}\label{lknxn}
Consider the subset ${\cal C}^+_{n,C}$ of ${\cal C}^+_n$ consisting in all vectors whose absolute values of the components are bounded by a real number $C$. Consider $N\in \N$ such that $N\leq n$. For each $j=0,\ldots,N-1$, let us introduce the vector $h^+_j=\left(h^+_j[1],\cdots,h^+_j[n]\right)'$ of $\R^n$ as follows 
\[
h^+_j[i] = \left \{ \begin{array}{lll} 0 & \textrm{ if }  i \leq \lfloor\frac{jn}{N}\rfloor \\
1 & \textrm{ otherwise }& \end{array} \right .  
\]
\noindent and define 
$$H_+=\Ve(h^+_0,\cdots,h^+_{N-1}).$$
Finally, set $\delta=2\sqrt{2}C/\sqrt{N}\geq 2\sqrt{2}C/\sqrt{n}$.

\begin{lem}\label{lem:nick}
With the previous notations, we have for all $f$ in ${\cal C}^+_{n,C}$
$$ \inf_{h\in H_+}\| f - h \|_n \leq \delta.$$
\end{lem}

\noindent \begin{proof}
We denote $f=(f[1],\ldots,f[n])'$, with
$$-C\leq f[1]\leq \dots\leq f[n]\leq C.$$
Set $\alpha_N=f[n]$ and, for $j=0,\dots,N-1$,
$$\alpha_j=\min_{i:h_j^+[i]=1}f[i]$$ 
We define also the vectors $f_-$ and $f_+$ of $H_+$ as follows
$$ f_-=\alpha_0h^+_0+\sum_{j=1}^{N-1}(\alpha_j-\alpha_{j-1})h^+_j$$
and 
$$ f_+=\alpha_1h^+_0+\sum_{j=1}^{N-1}(\alpha_{j+1}-\alpha_j)h^+_{j}.$$
Then we note that $f_-\leq f\leq f_+$, so that
$$\|f-f_-\|^2_n\leq\|f_+-f_-\|^2_n$$
with
\begin{equation}\label{lqBJAC} 
f_+-f_- = \sum_{j=1}^{N-1}(\alpha_j-\alpha_{j-1})(h^+_{j-1}-h^+_j)+(\alpha_N-\alpha_{N-1})h^+_{N-1}.
\end{equation} 
Remark that, for all $j=1,\dots,N-1$,
$$\|h^+_{j-1}-h^+_j\|^2_n\leq \frac{1}{n}\left(\lfloor\frac{jn}{N}\rfloor-\lfloor\frac{(j-1)n}{N}\rfloor\right)\leq\frac{1}{n}\left(\frac{n}{N}+1\right)\leq\frac{2}{N},$$
and $\|h^+_{N-1}\|^2_n\leq 2/N$ as well. Thus, taking into account that the decomposition (\ref{lqBJAC}) is orthogonal, we get
$$\|f_+-f_-\|^2_n\leq \frac{2}{N}\sum_{j=1}^{N}(\alpha_j-\alpha_{j-1})^2=\frac{8C^2}{N}\sum_{j=1}^{N}\left(\frac{\alpha_j-\alpha_{j-1}}{2C}\right)^2.$$
Since $0\leq (\alpha_j-\alpha_{j-1})/(2C)\leq 1$ and $0\leq (\alpha_N-\alpha_{1})/ 2C\leq1$, we have
$$\|f_+-f_-\|^2_n\leq \frac{8C^2}{N} \sum_{j=1}^N\frac{\alpha_j-\alpha_{j-1}}{2C}\leq\frac{8 C^2}{N}.$$
Considering that $\delta^2 =8 C^2/N$, we finally get the desired result, that is
$$\inf_{h\in H_+}\|f-h\|^2_n\leq \delta^2.$$    
\end{proof}

\noindent For the subset ${\cal C}^-_{n,C}$ of ${\cal C}^-_n$, we proceed in the same way. We conclude that there exists a vector space $H_-$ with dimension $N=8C^2/\delta^2$ such that, for all $f$ in ${\cal C}^-_{n,C}$,
$$ \inf_{h\in H^-}\| f - h \|_n \leq \delta.$$

\noindent\textbf{Acknowledgments.} We wish to thank Dragi Anevski and Enno Mammen to have made us aware of reference \cite{anevski2011monotone}. Arnaud Guyader is greatly indebted to Bernard Delyon for fruitful discussions on Von Neumann's algorithm.

\end{document}